\documentclass[a4paper,10pt]{amsart}

\usepackage{geometry}
\usepackage{amsthm}
\usepackage{amssymb}
\usepackage{amsmath}
\usepackage{hyperref}
\usepackage{xcolor}
\usepackage{enumerate}
\usepackage{listings}
\usepackage{complexity}
\usepackage{blkarray}
\usepackage{braket}
\usepackage[english]{babel}

\usepackage[font=small,labelfont=bf]{caption}
\usepackage[font=small,labelfont=normalfont,labelformat=simple]{subcaption}
\newtheorem{theorem}{Theorem}
\newtheorem{definition}{Definition}
\newtheorem{lemma}[theorem]{Lemma}
\newtheorem{corollary}[theorem]{Corollary}

\newtheorem{observation}{Observation}
\newtheorem{proposition}{Proposition}
\newtheorem{conjecture}{Conjecture}
\newtheorem{problem}{Problem}

\newcommand{\bivec}[1]{\overset{\text{\tiny$\longleftrightarrow$}}{#1}}

\author
{
Raphael Steiner
}
\thanks{Department of Computer Science, Institute of Theoretical Computer Science, ETH Z\"{u}rich, Switzerland,  \texttt{raphaelmario.steiner@inf.ethz.ch}. The author was supported by an
ETH Zurich Postdoctoral Fellowship.}

\date{\today}

\title{Coloring hypergraphs with excluded minors}

\begin{document}
\maketitle

\begin{abstract}
Hadwiger's conjecture, among the most famous open problems in graph theory, states that every graph that does not contain $K_t$ as a minor is properly $(t-1)$-colorable. 

The purpose of this work is to demonstrate that a natural extension of Hadwiger's problem to hypergraph coloring exists, and to derive some first partial results and applications.

Generalizing ordinary graph minors to hypergraphs, we say that a hypergraph $H_1$ is a minor of a hypergraph $H_2$, if a hypergraph isomorphic to $H_1$ can be obtained from $H_2$ via a finite sequence of the following operations: 

\begin{itemize}
    \item deleting vertices and hyperedges,
    \item contracting a hyperedge (i.e., merging the vertices of the hyperedge into a single vertex).
\end{itemize}
\smallskip
First we show that a weak extension of Hadwiger's conjecture to hypergraphs holds true: For every $t \ge 1$, there exists a finite (smallest) integer $h(t)$ such that every hypergraph with no $K_t$-minor is $h(t)$-colorable, and we prove
$$\left\lceil\frac{3}{2}(t-1)\right\rceil \le h(t) \le 2g(t)$$ where $g(t)$ denotes the maximum chromatic number of graphs with no $K_t$-minor. Using the recent result by Delcourt and Postle that $g(t)=O(t \log \log t)$, this yields $h(t)=O(t \log \log t)$.

We further conjecture that $h(t)=\left\lceil\frac{3}{2}(t-1)\right\rceil$, i.e., that every hypergraph with no $K_t$-minor is $\left\lceil\frac{3}{2}(t-1)\right\rceil$-colorable for all $t \ge 1$, and prove this conjecture for all hypergraphs with independence number at most $2$. 

By considering special classes of hypergraphs, the above additionally has some interesting applications for ordinary graph coloring, such as:

\begin{itemize}
    \item every graph $G$ is $O(k t \log \log t)$-colorable or contains a $K_t$-minor model all whose branch-sets are $k$-edge-connected, 
    \item every graph $G$ is $O( q t \log \log t )$-colorable or contains a $K_t$-minor model all whose branch-sets are modulo-$q$-connected (i.e., every pair of vertices in the same branch-set has a connecting path of prescribed length modulo $q$),
    \item by considering cycle hypergraphs of digraphs, we obtain known results on strong minors in digraphs with large dichromatic number as special cases. We also construct digraphs with dichromatic number $\left\lceil\frac{3}{2}(t-1)\right\rceil$ not containing the complete digraph on $t$ vertices as a strong minor, thus answering a question by M\'{e}sz\'{a}ros and the author in the negative.
\end{itemize}
\end{abstract}

\section{Preliminaries}

In this short first section we introduce essential terminology and notation used throughout the paper. The reader familiar with hypergraphs may want to skip this technical section and consult it at a later stage should there be any unclarities.

A \emph{hypergraph} is a tuple $(V,E)$ consisting of a finite set $V$ of vertices and a set $E \subseteq 2^V\setminus\{\emptyset\}$ of hyperedges.  Throughout this paper, all hypergraphs and graphs considered are assumed to have only hyperedges of size at least $2$, i.e., they are loopless. We refer to a hyperedge $e \in E$ as a \emph{graph edge} if it is of size $2$.

We denote $V(H):=V$, $E(H):=E$ for the vertex- and edge-set of a hypergraph. Given two hypergraphs $H_1$ and $H_2$, we say that $H_1$ is a \emph{subhypergraph of $H_2$}, in symbols, $H_1 \subseteq H_2$, if $V(H_1) \subseteq V(H_2)$ and $E(H_1)\subseteq E(H_2)$. We further say that $H_1$ is a \emph{proper} subhypergraph if at least one of the two inclusions above is strict. 

Given a hypergraph $H$ and a set $W \subseteq V(H)$, we denote by $H[W]$ the \emph{subhypergraph of $H$ induced by $W$}, whose vertex-set is $W$ and whose edge-set is $\{e \in E(H)|e \subseteq W\}$. For a hypergraph $H$ and an edge $e \in E(H)$ we write $H-e:=(V(H),E(H)\setminus \{e\})$ for the subhypergraph obtained by deleting $e$. Similarly, for a vertex $v \in V(H)$ we denote by $H-v:=H[V(H)\setminus\{v\}]$ the induced subhypergraph obtained by deleting $v$ and all the hyperedges containing $v$. For deleting sets $W\subseteq V(H)$ of vertices or $F \subseteq E(H)$ of hyperedges we use the analogous notation $H-W:=H[V(H)\setminus W], H-F:=(V(H),E(H)\setminus F)$.

We say that a vertex subset $W$ is \emph{independent} in $H$, if $H[W]$ contains no hyperedges, and we denote by $\alpha(H)$ the maximum size of an independent set in $H$. 

A hypergraph $H$ is \emph{connected} if for every partition of its vertex-set into non-empty subsets $X$ and $Y$, there exists a hyperedge $e$ in $H$ with $e \cap X\neq\emptyset\neq e \cap Y$. Equivalently, $H$ is connected if for every pair $x, y$ of distinct vertices there exists $k \ge 2$, a sequence of vertices $x_1,\ldots,x_k$ in $H$ and a sequence $e_1,\ldots,e_{k-1}$ of hyperedges in $H$ such that $x=x_1, y=x_k$ and $x_i, x_{i+1} \in e_i$ for all $i \in [k-1]$.
We say that a subset $W$ of $V(H)$ is \emph{connected} in $H$ if the induced subhypergraph $H[W]$ is connected.

In this manuscript, as usual, a \emph{proper coloring} of a hypergraph with a color-set $S$ is defined as a mapping $c:V(H) \rightarrow S$ such that for any $s \in S$ the color-class $c^{-1}(s)$ is independent in $H$. In other words, we assign colors from $S$ to the vertices of $H$ such that no hyperedge is monochromatic. The \emph{chromatic number} $\chi(H)$ is the minimum possible size of a color-set that can be used for a proper coloring of $H$. Note that $\chi(H)$ coincides with the ordinary chromatic number when $H$ is a graph. For an integer $k$, a hypergraph $H$ is called \emph{$k$-color-critical} if $\chi(H)=k$ but $\chi(H')<\chi(H)$ for every proper subhypergraph $H'$ of $H$.

A \emph{fractional coloring} of a hypergraph $H$ is an assignment of real-valued weights $w(I) \ge 0$ to all independent vertex sets $I$ in $H$ such that $\sum_{v \in I}{w(I)} \ge 1$ holds for every $v\in V(H)$. The \emph{weight} of a fractional coloring is defined to be the total weight $\sum{w(I)}$, where the sum is taken over all independent sets in $H$. Finally, the \emph{fractional chromatic number} of a hypergraph $H$, denoted by $\chi_f(H)$, is the smallest total weight a fractional coloring of $H$ can achieve (this infimum is always attained). We refer to~\cite{araujo} for an example of previous usage of the fractional chromatic number of hypergraphs.

A hypergraph is called \emph{Sperner} if it does not contain distinct hyperedges $e,f$ such that $e \subset f$. Given a hypergraph $H$, we denote by $\text{min}(H)$ the subhypergraph on the same vertex-set as $H$, whose hyperedges are exactly the inclusion-wise minimal hyperedges of $H$, i.e., $e \in E(\text{min}(H))$ if and only if $e\in E(H)$ and there exists no hyperedge $f \in E(H)\setminus\{e\}$ with $f \subset e$. It is easy to check that a set of vertices is independent in $H$ if and only if it is independent in $\text{min}(H)$. Also, it is clear by definition that $\text{min}(H)$ is Sperner. This immediately implies the following.
\begin{observation}\label{obs:subsperner}
For every hypergraph $H$, the subhypergraph $\text{min}(H)\subseteq H$ is Sperner and satisfies $\alpha(H)=\alpha(\text{min}(H))$, $\chi(H)=\chi(\text{min}(H))$ and $\chi_f(H)=\chi_f(\text{min}(H))$.
\end{observation}

\section{Introduction}

 Recall that given a graph $G_1$, another graph $G_2$ is a \emph{minor} of $G_1$ if $G_2$ is isomorphic to at least one graph that can be obtained from $G_1$ through a finite sequence of the following operations:

\begin{itemize}
\item moving to a subgraph,
\item contracting an edge (i.e., identifying its endpoints into a common vertex, and identifying parallel edges created by this process afterwards).
\end{itemize}

Maybe the most important open problem in graph coloring remains Hadwiger's conjecture, claiming the following relationship between minor-containment and the chromatic number:

\begin{conjecture}[Hadwiger 1943~\cite{hadwiger}]
For every integer $t \ge 2$, if $G$ is a graph which does not contain $K_t$ as a minor, then $\chi(G) \le t-1$.
\end{conjecture}

In this paper, we develop and study a hypergraph analogue of Hadwiger's conjecture. For this purpose, we generalize the definition of a graph minor to hypergraph minors in (what seems to the author) the most straightforward way as follows:

\begin{definition}
Let $H_1$ and $H_2$ be hypergraphs. We say that $H_2$ is a \emph{minor} of $H_1$ if $H_2$ is isomorphic to at least one hypergraph that can be obtained from $H_1$ via a finite sequence of the following operations:
\begin{itemize}
\item moving to a subhypergraph,
\item contracting a hyperedge (i.e., identifying all vertices contained in the hyperedge and removing loops or parallel hyperedges created by this process afterwards). 
\end{itemize}
\end{definition}

To avoid confusion, let us give a more formal description of a hyperedge-contraction in the following: Let a hypergraph $H$ and $e \in E(H)$ be given. Let $v \notin V(H)$ be a ``new'' vertex, which will serve as the contraction vertex for $e$. Define a map $\phi:V(H) \rightarrow (V(H)\setminus e) \cup \{v\}$ as follows: $\phi(x):=x$ for every $x \in V(H)\setminus e$, and $\phi(x):=v$ for every $x \in e$. We may now define the new hypergraph $H/e$ obtained by contracting $e$ as having vertex-set $V(H/e):=(V(H)\setminus e) \cup \{v\}$ and edge-set $E(H/e):=\{\phi(f)|f  \in E(H), f \not\subseteq e\}$. Note that, if we start from a hypergraph $H$ which is Sperner, the contracted hypergraph $H/e$ may no longer be Sperner.

We remark that various other definitions of minors for hypergraphs and simplicial complexes have been considered in the literature that are not addressed here, see for instance~\cite{adler,carmesin,nevo,wagner}.

For some of our proofs, it will be convenient to take a slightly different view of graph and hypergraph minors using so-called \emph{branch-sets} and \emph{minor models}. We extend these notions (which are well-established in graph minor theory) to hypergraphs as follows: 
\begin{definition}
Let $H_1$ and $H_2$ be hypergraphs. An \emph{$H_2$-minor model} in $H_1$ is a collection $(B_h)_{h \in V(H_2)}$ of disjoint non-empty sets of vertices in $H_1$, such that:
\begin{itemize}
    \item for every $h \in V(H_2)$, the set $B_h$ is connected in $H_1$, and
    \item for every hyperedge $f \in E(H_2)$, there is a hyperedge $e \in E(H_1)$ such that $e \subseteq \bigcup_{h \in f}{B_h}$ and $e \cap B_h \neq \emptyset$ for every $h \in f$.
\end{itemize}
The sets $B_h, h \in V(H_2)$ are called the \emph{branch-sets} of the $H_2$-minor model. 
\end{definition}
It is easy to see from the definition that if a hypergraph $H_1$ contains an $H_2$-minor model, then by contracting in $H_1$ all the edges within the connected branch-sets (in arbitrary order) and potentially deleting some superfluous edges and vertices afterwards, we obtain $H_2$ as a minor of $H_1$. Similarly, it is not hard to see and left for the reader to verify that if $H_1$ contains $H_2$ as a minor, then $H_1$ contains an $H_2$-minor model. 

\medskip

\paragraph{\textbf{New results.}} The main insight of this paper is that the above notion of hypergraph minors harmonizes well with the established notion of the hypergraph chromatic number $\chi(H)$, and allows for extending known partial results for Hadwiger's conjecture to hypergraph coloring. To discuss these statements more formally, we use the following notation: Given an integer $t \ge 2$, we denote by $g(t)$ the maximum chromatic number of a graph not containing $K_t$ as a minor, and similarly we denote by $h(t)$ the largest chromatic number of a hypergraph not containing $K_t$ as a minor. Note that at first glance it is far from obvious why the function $h(t)$ should be well-defined for every value of $t$. In our first main result below, we show that $h(t)$ exists for every $t \ge 2$ and, quite surprisingly, is tied to the graph function $g(t)$ by a factor of at most $2$.

\begin{theorem}\label{thm:main}
Every $K_t$-minor free hypergraph is $2g(t)$-colorable. Equivalently, $h(t) \le 2g(t)$.
\end{theorem}

The same proof idea as for Theorem~\ref{thm:main} also works for fractional coloring, yielding the following analogous result.

\begin{proposition}\label{prop:fractional}
Let $t \ge 2$ be an integer, and let $g_f(t)$ be the supremum of the fractional chromatic numbers taken over all $K_t$-minor free graphs. Then every $K_t$-minor free hypergraph $H$ satisfies $\chi_f(H) \le 2 g_f(t)$.
\end{proposition}

Hadwiger's conjecture has been proved for $t \le 6$, the case $t=6$ was settled by Robertson, Seymour and Thomas~\cite{robertson}. Thus, $g(t)=t-1$ for $t \le 6$. Asymptotically, the best known bound is that $g(t)=O(t \log\log t)$ as proved by Delcourt and Postle~\cite{del}. Finally, it was proved by Reed and Seymour~\cite{reedseymour} that $g_f(t)\le 2(t-1)$. Combining these known partial results for Hadwiger's conjecture with Theorem~\ref{thm:main} and Proposition~\ref{prop:fractional}, we get the following set of immediate consequences.

\begin{corollary}\label{smallbounds}
Let $t \ge 2$ be an integer, and let $H$ be a hypergraph without a $K_t$-minor. Then the following hold. 
\begin{itemize}
    \item if $t \in \{2,3,4,5,6\}$, then $\chi(H)  \le 2t-2$.
    \item for $t\ge 3$, we have $\chi(H) \le Ct \log\log t$, where $C>0$ is some absolute constant.
    \item $\chi_f(H) \le 4t-4$.
\end{itemize}
\end{corollary}

Given that $h(t)$ exists for every integer $t \ge 2$, it is tempting to hope or conjecture that Hadwiger's conjecture generalizes one-to-one to hypergraphs, in the sense that $h(t)=t-1$ for every $t \ge 2$. However, a more careful thought shows that this is not the case, and that in fact $h(t) \ge \left\lceil\frac{3}{2}(t-1)\right\rceil$ for every integer $t \ge 2$.

\begin{observation}\label{obs:lower}
For $t \ge 2$ the complete $3$-uniform hypergraph $H=K_{3(t-1)}^{(3)}$ does not contain $K_t$ as a minor and has chromatic number $\chi(H)=\left\lceil\frac{3}{2}(t-1)\right\rceil$. Thus, $h(t) \ge \left\lceil\frac{3}{2}(t-1)\right\rceil$.
\end{observation}
\begin{proof}
A set of vertices is independent in $H$ if and only if it has size at most $2$, which immediately yields $\chi(H)=\left\lceil\frac{3}{2}(t-1)\right\rceil$. Next suppose towards a contradiction that $H$ contains $K_t$ as a minor. Then $H$ must contain a $K_t$-minor model consisting of $t$ disjoint and connected branch-sets $(B_i)_{i=1}^{t}$, such that for every distinct $i, j \in [t]$ there is a hyperedge $e$ of $H$ contained in $B_i \cup B_j$ with $e \cap B_i \neq \emptyset \neq e \cap B_j$. Note that no branch-set $B_i$ can be of size $2$, since no such set is connected in $H$. Furthermore, there can be at most one branch-set of size $1$: If there were distinct $i, j$ with $|B_i|=|B_j|=1$, then $B_i \cup B_j$ would be too small to host a hyperedge of $H$. Consequently $$3(t-1)=|V(H)|\ge \sum_{i=1}^{t}{|B_i|} \ge 3(t-1)+1,$$ a contradiction, and this concludes the proof.
\end{proof}

Despite a significant effort, we were not able to find examples providing better lower bounds on $h(t)$ than the one given by the previous proposition. 
This finally leads us towards a hypergraph analogue of Hadwiger's conjecture as follows:

\begin{conjecture}\label{con:ours}
For every integer $t \ge 2$ we have $h(t)=\left\lceil\frac{3}{2}(t-1)\right\rceil$. In other words, if a hypergraph $H$ does not contain $K_t$ as a minor, then $\chi(H) \le \left\lceil\frac{3}{2}(t-1)\right\rceil$. 
\end{conjecture}

Intriguingly,Conjecture~\ref{con:ours} remains open even for $K_3$-minor-free hypergraphs. The best upper bound on the chromatic number of $K_3$-minor free hypergraphs we are aware of is $4$, provided by Corollary~\ref{smallbounds}. 

\begin{problem}
Is every $K_3$-minor-free hypergraph $3$-colorable?
\end{problem}

We remark that independently of our work, van der Zypen~\cite{vanderZypen} has stated an extension of Hadwiger's conjecture to hypergraphs. However, his alternative conjecture uses a much less restricted version of hypergraph minors, hence leading to a more restricted class of hypergraphs with no $K_t$-minor. For instance, according to the definition in~\cite{vanderZypen}, every hypergraph consisting of a single hyperedge on $t$ vertices would contain $K_t$ as a minor. An argument outlined in~\cite{vanderZypen2} shows that for \emph{finite} values of $t$ van der Zypen's conjecture is a consequence of the ordinary Hadwiger's conjecture for graphs. Since our focus in this paper is on finite hypergraphs, we will not discuss this variant in more detail here.

As additional evidence for our Conjecture~\ref{con:ours} we verify it for all hypergraphs without $3$-vertex independent sets. This is a natural special case to look at, due to the considerable attention that the special case of Hadwiger's conjecture for graphs with independence number $2$ has received in the past, see e.g.~\cite{blasiak,bosse,chudnovsky,duchetmeyniel,furedi,kriesell,plummer}.

\begin{theorem}\label{thm:indep2}
Let $t \ge 2$ be an integer, and let $H$ be a hypergraph such that $\alpha(H) \le 2$. If $H$ does not contain $K_t$ as a minor, then $\chi(H) \le \left\lceil\frac{3}{2}(t-1)\right\rceil$. 
\end{theorem}


\paragraph{\textbf{Applications.}} We present three simple examples of how our bounds for coloring $K_t$-minor free hypergraphs can be applied to produce new or recover known results for coloring graphs and directed graphs with excluded minor-like substructures. We believe that more applications of a similar flavour will be found in the future.

Our first two applications address the following natural question. While Hadwiger's conjecture guarantees the existence of a $K_t$-minor model with connected branch-sets in graphs with high chromatic number, in principle these are not very dense structures, as every branch-set could span a tree only. However, intuitively, graphs of high chromatic number should also contain $K_t$-minor models with ``richer'' branch-sets (i.e., denser, of higher connectivity, etc). Our first result in this direction confirms this intuition by proving that a very moderate bound on the chromatic number suffices to guarantee high edge-connectivity in the branch-sets. 

\begin{theorem}\label{thm:kconn}
There exists an absolute constant $C>0$ such that for every pair of integers $k \ge 1, t \ge 3$ every graph $G$ satisfying $\chi(G)>Ckt\log\log t$ contains a $K_t$-minor model with branch-sets $(B_i)_{i=1}^{t}$, such that for each $i \in [t]$ the induced subgraph $G[B_i]$ is $k$-edge-connected. Furthermore, for any distinct $i, j \in [t]$ there are $k$ distinct edges in $G$ connecting $B_i$ and $B_j$.
\end{theorem}

In our next application of Theorem~\ref{thm:main} we consider another way of strengthening the connectivity requirement for the branch-sets of a complete minor, by requiring that distinct vertices in the same branch-set can be connected by a path of any desired length modulo $q$. For $q=2$, this is somewhat reminiscent of (but not equivalent to) the recently popular notion of \emph{odd minors}. The interested reader may consult e.g.~\cite{geelen} for a definition and background.

\begin{definition}
Let $q \ge 2$ be an integer and $G$ a graph. We say that $G$ is \emph{modulo $q$-connected} if for every pair of distinct vertices $u, v \in V(G)$ and every residue $r \in \{0,1,\ldots,q-1\}$ there exists a path $P$ in $G$ with endpoints $u, v$ and length $\ell(P) \equiv_q r$. 
\end{definition}

\begin{theorem}\label{thm:qconn}
There exists an absolute constant $c>0$ such that for every pair of integers $q\ge 2$, $t \ge 3$ every graph $G$ satisfying $\chi(G)>Cqt\log\log t$ contains a $K_t$-minor model with branch-sets $(B_i)_{i=1}^{t}$, such that for each $i \in\{1,\ldots,t\}$ the induced subgraph $G[B_i]$ is modulo-$q$-connected. 
\end{theorem}

In our third and last application, we give a short reproof of the main result from~\cite{tamas} on coloring digraphs with excluded strong complete minors. The research on this topic was initiated by Axenovich, Gir\~{a}o, Snyder and Weber~\cite{axenovich}, and then further addressed by M\'{e}sz\'{a}ros and the author~\cite{tamas}. To state the result, we adopt the following terminology from~\cite{axenovich,tamas}:

For digraphs $D$ and $F$, we say that \emph{$D$ contains $F$ as a strong minor}, if there exist disjoint non-empty subsets $(B_f)_{f \in V(F)}$ of $V(D)$ such that $B_f$ induces a strongly connected subdigraph of $D$ for every $f \in V(F)$ and for every arc $(f_1,f_2)$ in $F$, there is an arc from $B_{f_1}$ to $B_{f_2}$ in $D$. 
The \emph{dichromatic number} $\vec{\chi}(D)$ of a digraph $D$ is the smallest number of colors that can be used to color the vertices of $D$ such that every color class spans an acyclic subdigraph of $D$. 

The following is a main result from~\cite{tamas}. By considering \emph{cycle hypergraphs} of digraphs, we will show that it is an immediate consequence of Theorem~\ref{thm:main}. For $t \ge 1$ we denote by $\bivec{K_t}$ the complete digraph of order $t$ (containing all possible $t(t-1)$ arcs).

\begin{theorem}\label{thm:strong}
If $D$ is a digraph, $t \ge 2$ an integer, and $D$ does not contain $\bivec{K_t}$ as a strong minor, then $D$ has dichromatic number at most $h(t)\le 2g(t)$.
\end{theorem}

It was raised as an open problem in~\cite{tamas} whether or not every digraph $D$ with no $\bivec{K_t}$ strong minor has dichromatic number at most $t$. Inspired by our lower bound on $h(t)$ from Proposition~\ref{obs:lower}, we will actually show that this is \emph{not} the case: 
\begin{proposition}\label{prop:digraphlower}
For every $t \ge 2$ there is a digraph $D$ with no strong $\bivec{K_t}$-minor and dichromatic number at least $\left\lceil\frac{3}{2}(t-1)\right\rceil$.
\end{proposition}

\paragraph{\textbf{Structure of the paper.}} In Section~\ref{sec:proofs} we give the proofs of our main results, including Theorem~\ref{thm:main}, Proposition~\ref{prop:fractional} and Theorem~\ref{thm:indep2}. In Section~\ref{sec:applications} we give the formal proofs of the three applications of these main results to graph and hypergraph coloring, including the proofs of Theorem~\ref{thm:kconn}, Theorem~\ref{thm:qconn}, Theorem~\ref{thm:strong} and Proposition~\ref{prop:digraphlower}. 
 
\section{Proofs of Theorem~\ref{thm:main}, Proposition~\ref{thm:main} and Theorem~\ref{thm:indep2}}\label{sec:proofs}

We start by giving the (joint\footnote{We decided to merge the proofs of these results into one, since they are based on largely the same idea.}) proof of Theorem~\ref{thm:main} and Proposition~\ref{prop:fractional}. The idea of the proof is to decompose the hypergraph into connected and $2$-colorable pieces and derive a graph from that partition. 
\begin{proof}[Proof of Theorem~\ref{thm:main} and Proposition~\ref{prop:fractional}]
Let $t \ge 2$ be an integer, and $H$ be a hypergraph with no $K_t$-minor. Our goal is to show that $\chi(H) \le 2g(t)$, and $\chi_f(H) \le 2g_f(t)$. 

W.l.o.g.~we may assume that $H$ is Sperner, for if not then by Observation~\ref{obs:subsperner} the subhypergraph $\text{min}(H)$ is Sperner, also contains no $K_t$-minor and has the same chromatic and fractional chromatic number as $H$.

Let us inductively construct a partition of the vertex-set $V(H)$ into non-empty sets $(X_i)_{i=1}^{n}$ for some number $n$ as follows: For $i \ge 1$, as long as $X_1 \cup \cdots \cup X_{i-1} \neq V(H)$, pick $X_i$ as a subset of $V(H) \setminus (X_1 \cup \cdots \cup X_{i-1})$, such that $X_i$ is inclusion-wise maximal with respect to the following two properties:
\begin{itemize}
\item $H[X_i]$ is connected, and 
\item $\chi(H[X_i]) \le 2$. 
\end{itemize}
Note that such a set always exists, since any singleton set in $V(H) \setminus (X_1 \cup \cdots \cup X_{i-1})$ satisfies both of the above properties. 
We now define a graph $G$ by $V(G):=[n]$ and $$E(G):=\Set{\{i,j\} \in \binom{[n]}{2}\:|\:\text{there exists }e \in E(H) \text{ with }e \subseteq X_i \cup X_j\text{ and }e \cap X_i \neq \emptyset \neq e \cap X_j}.$$

Since $H[X_i]$ is connected for $i=1,\ldots,n$, by contracting (in arbitrary order) all hyperedges of $H$ contained in $H[X_i]$ for some $i \in [n]$, we obtain a hypergraph $\tilde{H}$ on $n$ vertices. It now follows directly from the definition of $G$ that for every edge $\{i,j\} \in E(G)$, the pair consisting of the two contraction vertices of $X_i$ and $X_j$ forms a hyperedge in $\tilde{H}$. Hence, after deleting all edges in $\tilde{H}$ which are not such pairs, we obtain a graph isomorphic to $G$. Thus, $G$ is a minor of $H$. 
Since $H$ does not contain $K_t$ as a minor, the same must be true for $G$, and it follows that $\chi(G) \le g(t)$ and $\chi_f(G) \le g_f(t)$.

\smallskip

\textbf{Claim 1.} Let $e \in E(H)$. If $I(e):=\{i \in [n]|e \cap X_i \neq \emptyset\}$ forms an independent set in $G$, then there exists $j \in [n]$ such that $e \subseteq X_j$.

\begin{proof}
Suppose towards a contradiction that the claim is wrong, and let $e \in E(H)$ be chosen such that $|I(e)|$ is minimized among all hyperedges for which the claim fails. 

Let $j:=\min I(e)$ be the (index-wise) smallest element of $I(e)$. Since $e$ does not satisfy the claim, we have $e \not\subseteq X_j$ and hence $Y:=e \setminus X_j \neq \emptyset$. By minimality of $j$ we have $Y \subseteq X_{j+1} \cup \cdots \cup X_n$. If $I(e)=\{j,j'\}$ for some $j'>j$, then by definition of $I(e)$ we have $e \subseteq X_j \cup X_{j'}$ and $e \cap X_j \neq \emptyset \neq e \cap X_{j'}$. The definition of $G$ now yields $\{j,j'\} \in E(G)$, contradicting the fact that $I(e)$ is an independent set in $G$. Moving on, we may therefore assume that $|I(e)| \ge 3$. Let $I_1, I_2$ be two non-empty disjoint sets such that $I(e) \setminus \{j\}=I_1 \cup I_2$. By definition, $X_j$ is inclusion-wise maximal among the subsets $X$ of $V(H) \setminus (X_1 \cup \cdots \cup X_{j-1})$ for which $H[X]$ is connected and $\chi(H[X]) \le 2$. We will now show that the strict superset $X_j \cup Y$ of $X_j$ also satisfies these two properties, a contradiction which then concludes the proof of Claim 1. 

It remains to show that $H[X_j \cup Y]$ is also connected and $2$-colorable. The connectivity follows directly from the facts that $H[X_j]$ is connected, $Y=e \setminus X_j$ and $e \cap X_j \neq \emptyset$. Let now $c:X_j \rightarrow \{1,2\}$ be a proper $2$-coloring of $H[X_j]$, and extend $c$ to a $2$-coloring $c'$ of $H[X_j \cup Y]$ by putting $c'(y):=1$ for every $y \in Y$ such that $y \in X_i$ for some $i \in I_1$, and $c'(y):=2$ for every $y \in Y$ such that $y \in X_i$ for some $i \in I_2$. Towards a contradiction, suppose that there is a monochromatic hyperedge $f$ in $H[X_j \cup Y]$ with respect to $c'$. Then since $f$ is monochromatic and $f \subseteq X_j \cup Y$, we conclude that $I(f) \subseteq \{j\} \cup I_1$ or $I(f) \subseteq \{j\} \cup I_2$. In each case, $I(f) \subsetneq I(e)$, and thus $I(f)$ is independent in $G$ and strictly smaller than $I(e)$. Now our initial assumption on the minimality of $|I(e)|$ tells us that $f$ satisfies the statement of Claim 1. We know however that this is not the case: $f \subseteq X_j$ would imply that $f$ is a monochromatic edge in the proper $2$-coloring $c$ of $H[X_j]$, which is impossible, and similarly $f \subseteq X_{i}$ for some $i \in [n]\setminus\{j\}$ would imply that $f \subseteq Y=e \setminus X_j$, a contradiction since $H$ is Sperner. All in all, this concludes the argument that $H[X_j \cup Y]$ is both connected and $2$-colorable, concluding the proof of Claim~1.
\end{proof} 
We can now argue that $\chi(H) \le 2\chi(G) \le 2g(t)$ and $\chi_f(H)\le 2\chi_f(G)\le2g_f(t)$, completing the proof of Theorem~\ref{thm:main}. 

To do so, for every $i \in [n]$, fix a proper $2$-coloring $c_i:X_i \rightarrow \{1,2\}$ of $H[X_i]$. 

To verify $\chi(H) \le 2\chi(G)$, consider a proper coloring $c_G:V(G) \rightarrow [\chi(G)]$ of $G$. Let $c_H:V(H) \rightarrow [\chi(G)] \times \{1,2\}$ be defined as $c_H(x):=(c_G(i),c_i(x))$ for every $x \in X_i$, $i \in [n]$. We claim that $c_H$ defines a proper coloring of $H$. Suppose not, then there exists $e \in E(H)$ (w.l.o.g. inclusion-wise minimal) such that all vertices in $e$ receive the same color under $c_H$, say $c_H(x)=(\alpha,\beta)$ for every $x \in e$. This shows that for every $i \in [n]$ such that $e \cap X_i \neq \emptyset$, we have $c_G(i)=\alpha$. Since $c_G$ is a proper coloring of $G$, this means that (with the notation from Claim~1) the set $I(e):=\{i \in [n]|e \cap X_i \neq \emptyset\}$ is independent in $G$. Hence, we may apply Claim~1 to $e$, which yields that there exists $j \in [n]$ such that $e \subseteq X_j$. We then have $c_j(x)=\beta$ for every $x \in e$, contradicting the facts that $e$ is a hyperedge of $H[X_j]$ and that $c_j$ was chosen as a proper coloring of $H[X_j]$. This contradiction shows that $c_H$ is indeed a proper coloring of $H$, yielding $\chi(H) \le 2\chi(G) \le 2g(t)$, as claimed. 

To verify $\chi_f(H) \le 2\chi_f(H)$, we proceed similarly. Denote by $\mathcal{I}(G)$, $\mathcal{I}(H)$ the collections of independent sets in $G$ and $H$, and let $(w_G(I))_{I \in \mathcal{I}(G)}$ be a non-negative weight assignment such that $\sum_{i \in I \in \mathcal{I}(G)}{w_G(I)} \ge 1$ for every $i \in [n]$ and $\sum_{I \in \mathcal{I}(G)}{w_G(I)}=\chi_f(G)$. Note that for every $i \in [n]$, the subsets $c_i^{-1}(1), c_i^{-1}(2)$ of $X_i$ are independent in $H$. Using Claim~1 this implies that for every $I \in \mathcal{I}(G)$, and any string $s \in \{1,2\}^I$, the set 
$$J(I,s):=\bigcup_{i \in I}{c_i^{-1}(s_i)}$$ is independent in $H$. 

We now define a weight assignment $w_H$ for $\mathcal{I}(H)$ by putting 
$w_H(J(I,s)):=2^{1-|I|}w_G(I)$ for every $I \in \mathcal{I}(G)$ and $s \in \{1,2\}^I$, and assigning weight $0$ to all remaining independent sets in $H$. We then have for every $i \in [n]$ and $v \in X_i$:
$$\sum_{v \in J \in \mathcal{I}(H)}{w_H(J)}=\sum_{i \in I \in \mathcal{I}(G)}\sum_{\substack{s \in \{1,2\}^I: \\ s_i=c_i(v)}}{2^{1-|I|}w_G(I)}$$ $$=\sum_{i \in I \in \mathcal{I}(G)}{2^{|I|-1}(2^{1-|I|}w_G(I))}=\sum_{i \in I \in \mathcal{I}(G)}{w_G(I)} \ge 1.$$
Therefore, $w_H$ is a legal weight assignment for $H$, and it follows that $$\chi_f(H)\le \sum_{J \in \mathcal{I}(H)}{w_H(J)}=\sum_{I \in \mathcal{I}(G)}{2^{|I|}(2^{1-|I|}w_G(I))}=2\sum_{I \in \mathcal{I}(G)}{w_G(I)}=2\chi_f(G),$$ as desired.
\end{proof}

We proceed by preparing the proof of Theorem~\ref{thm:indep2}. A crucial ingredient to that proof is the following lemma. We note that its statement restricted to graphs is well-known (compare~\cite{duchetmeyniel}), and that our proof follows similar lines as the one in the case of graphs.

\begin{lemma}\label{lemma:indep2}
Let $H$ be a hypergraph on $n$ vertices satisfying $\alpha(H)\le 2$. Then $H$ contains $K_{\lceil n/3\rceil}$ as a minor.
\end{lemma}
\begin{proof}
For every hypergraph $H$ with $\alpha(H)\le 2$ we may instead consider $\text{min}(H)$. By Observation~\ref{obs:subsperner} this subhypergraph satisfies $\alpha(\text{min}(H))=\alpha(H) \le 2$, is Sperner, and if it contains $K_{\lceil n/3\rceil}$ as a minor then the same is true for $H$. So in the remainder, we assume w.l.o.g. that $H$ is Sperner. Note that this implies that $H$ contains no $r$-edges for any $r >3$, since if it were to contain an $r$-edge $e$, removing any vertex $x \in e$ would result in the independent set $e\setminus\{x\}$ of size $r-1 > 2$ in $H$, which is impossible. 

In the following let us call a subset $X \subseteq V(H)$ of exactly $3$ vertices \emph{nice} if the induced subhypergraph $H[X]$ contains as hyperedges either exactly one hyperedge of size $3$ or exactly two graph edges. Let $\mathcal{X}$ be a maximal collection of pairwise disjoint nice sets in $H$ (here we allow $\mathcal{X}=\emptyset$). Let $U:=\bigcup\mathcal{X}$ denote the set of vertices contained in a member of $\mathcal{X}$. Then by maximality $V(H)\setminus U$ does not contain a nice set. This in particular means that $H-U$ does not contain any hyperedges of size more than $2$, so it is a graph, and also no induced $3$-vertex path. It is well-known (and not hard to verify) that the only graphs with no induced $3$-vertex paths are disjoint unions of complete graphs. Since $\alpha(H-U)\le\alpha(H)\le 2$, this means that $H-U$ is the disjoint union of at most $2$ complete graphs. Let us write $n=a+b$, where $a:=|U|$ and $b:=|V(H)\setminus U|$. From the above we find that there exists a subset $C \subseteq V(H)\setminus U$ of size $|C| \ge \frac{b}{2} \ge\frac{b}{3}$ such that $H[C]$ is a complete graph. 
We now claim that the members of $\mathcal{X}$ together with the singletons in $C$ form the branch-sets of a minor model for a $K_t$ of order $t:=|\mathcal{X}|+|C| \ge \frac{a}{3}+\frac{b}{3}=\frac{n}{3}$ in $H$, which will conclude the proof. 

To see this, note that every member of $\mathcal{X}$ (and trivially every singleton in $C$) induces a connected subhypergraph of $H$. Furthermore, for every nice set $X \in \mathcal{X}$ and every vertex $y$ of $H$ outside $X$, there exists an edge $e \in E(H)$ with $e \subseteq X \cup \{y\}$ and $y \in e$. Indeed, if such an edge would not exist, then it is easy to see that an independent set of size $2$ in $H[X]$ (which exists since $X$ is nice) joined by $y$ would form an independent set of size $3$ in $H$, a contradiction. 

The above implies that for any two distinct members $X, Y \in \mathcal{X}$ there exists a hyperedge in their union intersecting both $X$ and $Y$, and also that for every $X \in \mathcal{X}$ and $c \in C$ there is a hyperedge contained in $X \cup \{c\}$ intersecting both $X$ and $\{c\}$. Since additionally $C$ is a clique, this confirms our above claim that $\mathcal{X}$ together with $\{c\}_{c \in C}$ yields a complete minor in $H$. 

\end{proof}

To prove Theorem~\ref{thm:indep2}, we additionally need a structural result about small color-critical hypergraphs established by Stiebitz, Storch and Toft in~\cite{stiebitz}. We adopt the following terminology from their paper:

\begin{definition}
Given two hypergraphs $H_1=(V_1,E_1)$ and $H_2=(V_2,E_2)$ with disjoint vertex-sets $V_1$ and $V_2$, their \emph{Dirac-sum} $H_1+H_2$ is the hypergraph with the vertex-set $V=V_1 \cup V_2$ and the edge-set $E=E_1 \cup E_2 \cup \{\{x,y\}|x \in V_1, y \in V_2\}$. 
\end{definition}

We may now state their result as follows.

\begin{theorem}[Stiebitz, Storch and Toft~\cite{stiebitz}, Theorem~1]\label{thm:decompose}
Let $H$ be a $k$-color critical hypergraph on at most $2k-2$ vertices. Then there exist two non-empty disjoint subhypergraphs $H_1$ and $H_2$ of $H$ such that $H=H_1+H_2$.
\end{theorem}

We remark the following simple property of the Dirac sum for later use:

\begin{observation}\label{obs:chromaticsum}
$\chi(H_1+H_2)=\chi(H_1)+\chi(H_2)$ for all hypergraphs $H_1, H_2$ on disjoint vertex-sets.
\end{observation}

We are now sufficiently prepared for the proof of Theorem~\ref{thm:indep2}.

\begin{proof}[Proof of Theorem~\ref{thm:indep2}]
Towards a contradiction, suppose the claim is false and let $t \ge 2$ be the smallest integer for which there exists a counterexample. Let $H$ be a non-trivial hypergraph with $|V(H)|+|E(H)|$ minimum such that: 
\begin{itemize}
    \item $\alpha(H) \le 2$
    \item $H$ does not contain $K_t$ as a minor, and
    \item $\chi(H)>\left\lceil \frac{3}{2}(t-1) \right\rceil$.
\end{itemize}
We start with a few observations about $H$. 

\medskip

\textbf{Claim 1.} $H$ is Sperner. 

\begin{proof}[Subproof.]
If $H$ is not Sperner, then $\text{min}(H)$ is a proper subhypergraph of $H$, which by Observation~\ref{obs:subsperner} has the same chromatic and independence number as $H$, and clearly also does not contain $K_t$ as a minor. This is a contradiction to our minimality assumption on $H$.
\end{proof}

\smallskip

\textbf{Claim 2.} $H$ is $k$-color-critical, with $k:=\left\lceil \frac{3}{2}(t-1) \right\rceil+1$. 
\begin{proof}[Subproof]
By our assumptions on $H$ we have $\chi(H) \ge k$. Thus the claim follows if we can show that for every proper subhypergraph $H'\subsetneq H$ we have $\chi(H') \le k-1=\left\lceil \frac{3}{2}(t-1) \right\rceil$. Clearly, this is equivalent to showing that $H-v$ and $H-e$ for every choice of $v \in V(H)$ and $e \in E(H)$ are $\left\lceil \frac{3}{2}(t-1) \right\rceil$-colorable. 
Note for every $v \in V(H)$ that $H-v$ is an induced subhypergraph of $H$, which directly implies that $\alpha(H-v) \le 2$. Trivially, also $H-v$ does not contain $K_t$ as a minor, which means that we must have $\chi(H-v) \le \left\lceil \frac{3}{2}(t-1) \right\rceil$, for otherwise $H-v$ would be a counterexample to the claim of the theorem with $|V(H-v)|+|E(H-v)|<|V(H)|+|E(H)|$, contradicting the minimality of $H$. 

Let now $e \in E(H)$ be arbitrary. To argue that $H-e$ is $\left\lceil\frac{3}{2}(t-1) \right\rceil$-colorable, we have to be a bit more careful: Deleting the hyperedge $e$ could increase the independence number, such that $H-e$ may fall out of the class of hypergraphs with independence number $2$. We avoid this obstacle by instead considering the hypergraph $H':=H/e$, obtained from $H$ by contracting $e$. Then it is not hard to see that $\alpha(H') \le \alpha(H)=2$. Further note that $H'$ as a minor of $H$ cannot contain $K_t$ as a minor. Now the minimality of $H$ as a counterexample and the fact $|V(H')|+|E(H')|<|V(H)|+|E(H)|$ imply that $\chi(H') \le \left\lceil \frac{3}{2}(t-1) \right\rceil$. Thus there exists a proper coloring $c':V(H') \rightarrow \{1,\ldots,\left\lceil \frac{3}{2}(t-1) \right\rceil\}$ of $H'$. Let $c:V(H)\rightarrow \{1,\ldots,\left\lceil \frac{3}{2}(t-1) \right\rceil\}$ be the coloring defined by $c(x):=c'(\phi(x))$ for every $x \in V(H)$, where $\phi:V(H)\rightarrow V(H/e)$ denotes the ``contraction map'' as defined in the preliminaries. We claim that $c$ is a proper coloring of the hypergraph $H \setminus e$. Indeed, consider any edge $f \in E(H)\setminus \{e\}$. Then by Claim~1, we cannot have $f \subseteq e$, and hence the image $\phi(f)$ (by definition of $H/e$) forms a hyperedge in $H/e$. If $f$ were to be monochromatic with respect to the coloring $c$, then its image $\phi(f)$ would be monochromatic with respect to the coloring $c'$ of $H'$. Since this is impossible, we conclude that indeed $H\setminus e$ is properly colored by $c$, proving that $\chi(H\setminus e)\le \left\lceil \frac{3}{2}(t-1) \right\rceil=k-1$ for every $e \in E(H)$. This concludes the proof of Claim~2.
\end{proof}
In the following denote $n:=|V(H)|$. By Lemma~\ref{lemma:indep2} applied to $H$, we find that $H$ contains $K_{\lceil n/3\rceil}$ as a minor, which by our assumptions on $H$ implies that $t \ge \left\lceil \frac{n}{3}\right\rceil+1$. This yields $n \le 3(t-1)=2 \cdot \frac{3}{2}(t-1) \le 2(k-1)=2k-2$. Since by Claim~2 $H$ is $k$-color-critical, we may apply Theorem~\ref{thm:decompose} to $H$. This yields the existence of a pair $(H_1,H_2)$ of non-empty disjoint subhypergraphs of $H$ such that $H=H_1+H_2$. 



In the following, let us denote by $t_1$ and $t_2$ the largest integers such that $H_1$ contains $K_{t_1}$ as a minor, and such that $H_2$ contains $K_{t_2}$ as a minor, respectively. It is easy to see that any sequence of deletions and contractions resulting in the minor of $K_{t_1}$ of $H_1$ and $K_{t_2}$ of $H_2$ can be concatenated to yield a sequence of deletions and contractions in $H$. Performing this sequence of operations is then easily seen to result in the minor $K_{t_1}+K_{t_2}=K_{t_1+t_2}$ of $H$. By our assumptions on $H$, this implies that $t_1+t_2 \le t-1$. 

Since $H_1$ and $H_2$ are induced subhypergraphs of $H$, they satisfy $\alpha(H_1), \alpha(H_2) \le \alpha(H)\le 2$. 

Hence, by our minimality assumption on $t$ at the beginning of this proof, since $t_i+1 \le t_1+t_2<t$ for $i=1,2$, and since $H_i$ does not contain $K_{t_i+1}$ as a minor for $i=1,2$, we find that $\chi(H_i) \le \left\lceil\frac{3((t_i+1)-1)}{2}\right\rceil=\left\lceil\frac{3t_i}{2}\right\rceil$ for $i=1,2$. 

Now pick (arbitrarily) a pair of vertices $x \in V(H_1)$ and $y \in V(H_2)$. We distinguish two cases depending on the behaviour of the largest clique minors in $H_1-x$ and $H_2-y$. 

\medskip

\paragraph{\textbf{Case 1.}} $H_1-x$ contains $K_{t_1}$ as a minor, and $H_2-y$ contains $K_{t_2}$ as a minor. 

Then, repeating the argument from above, we find that $(H_1-x)+(H_2-y)=H-\{x,y\}$ contains $K_{t_1}+K_{t_2}=K_{t_1+t_2}$ as a minor. Also note that every vertex in $H$ is connected by a graph edge to either $x$ or $y$. Therefore, contracting the edge $\{x,y\}$, and then successively performing the deletions and contractions that transform $H-\{x,y\}$ into $K_{t_1+t_2}$, yields the complete graph $K_{t_1+t_2+1}$ as a minor of $H$. Since we assumed that $H$ does not contain $K_t$ as a minor, this implies that $t_1+t_2+1 \le t-1$. We may now use this to bound the chromatic number of $H$ as follows:

$$\chi(H)=\chi(H_1)+\chi(H_2) \le \left\lceil\frac{3t_1}{2}\right\rceil+\left\lceil\frac{3t_2}{2}\right\rceil \le \left(\frac{3t_1}{2}+\frac{1}{2}\right)+\left(\frac{3t_2}{2}+\frac{1}{2}\right)$$
$$=\frac{3(t_1+t_2+1)-1}{2}<\frac{3(t_1+t_2+1)}{2} \le \frac{3(t-1)}{2} \le \left\lceil\frac{3(t-1)}{2}\right\rceil,$$ which is a contradiction to our initial assumption $\chi(H)>\left\lceil\frac{3(t-1)}{2}\right\rceil$ on $H$. Hence, we obtain the desired contradiction, which concludes the proof of the theorem in Case~1. 

\medskip

\paragraph{\textbf{Case 2.}} $H_1-x$ does not contain $K_{t_1}$ as a minor, or $H_2-y$ does not contain $K_{t_2}$ as a minor. Due to symmetry, in the following we may assume without loss of generality that $H_1-x$ does not contain $K_{t_1}$ as a minor.

Then we can obtain a better estimate on the chromatic number of $H_1$: Since $H_1-x$ does not contain $K_{t_1}$ as a minor and since $\alpha(H_1-x)\le \alpha(H_1)\le 2$, by minimality of $t$ we obtain $\chi(H_1-x)\le \left\lceil\frac{3(t_1-1)}{2}\right\rceil$. Since the deletion of a vertex can lower the chromatic number of a hypergraph by at most
one, we conclude:

$$\chi(H)=\chi(H_1)+\chi(H_2)\le (\chi(H_1-x)+1)+\chi(H_2)$$
$$\le \left\lceil\frac{3(t_1-1)}{2}\right\rceil+\left\lceil\frac{3t_2}{2}\right\rceil+1$$

Consider first the case that both $t_1$ and $t_2$ are odd numbers, meaning that both $3t_1$ and $3t_2$ are odd, and both $3(t_1-1)$ and $3(t_2-1)$ are even. Then we may estimate, using the above:

$$\chi(H)\le\frac{3(t_1-1)}{2}+\frac{3t_2+1}{2}+1$$
$$= \frac{3(t_1+t_2)}{2}\le \frac{3(t-1)}{2} \le \left\lceil\frac{3(t-1)}{2}\right\rceil.$$ Again, this yields a contradiction to the assumption $\chi(H)>\left\lceil\frac{3(t-1)}{2}\right\rceil$ and concludes the proof in this case.

Finally, assume that at least one of $t_1$ and $t_2$ is even. Then we may estimate:

$$\chi(H)=\chi(H_1)+\chi(H_2) \le \left\lceil\frac{3t_1}{2}\right\rceil+\left\lceil\frac{3t_2}{2}\right\rceil\le \frac{3t_1}{2}+\frac{3t_2}{2}+\frac{1}{2}=\frac{3(t_1+t_2)}{2}+\frac{1}{2}$$
$$\le \frac{3(t-1)}{2}+\frac{1}{2}<\frac{3(t-1)}{2}+1\le \left\lceil\frac{3(t-1)}{2}\right\rceil +1.$$ This is a contradiction, since our initial assumption on $H$ means that $\chi(H) \ge \left\lceil\frac{3(t-1)}{2}\right\rceil+1$.

This concludes the proof of the theorem also in Case~2.
\end{proof}

\section{Applications}\label{sec:applications}

\subsection{Complete minors with branch-sets of high edge-connectivity}
As a first simple application of Theorem~\ref{thm:main} (or, rather Corollary~\ref{smallbounds}), we want to use it to give a short proof of Theorem~\ref{thm:kconn}. To reach this goal, we associate with every graph $G$ and every integer $k \ge 1$ a hypergraph $H_{\text{conn}}(G,k)$ as follows: Its vertex-set is $V(G)$, and a subset of vertices $e \subseteq V(G)$ of size at least $2$ forms a hyperedge of $H_{\text{conn}}(G,k)$ if and only if the induced subgraph $G[e]$ is $k$-edge-connected. The next lemma collects simple but important properties of this hypergraph.
\begin{lemma}\label{lemma:kconn}
Let $G$ be a graph and $k \ge 1$. Then
\begin{enumerate}
    \item If $W \subseteq V(G)$ is connected in $H_{\text{conn}}(G,k)$, then $G[W]$ is $k$-edge-connected. 
    \item If $\chi(G)>k$ then $H_{\text{conn}}(G,k)$ contains at least one hyperedge. 
    \item $\chi(G)\le k\chi(H_\text{conn}(G,k))$.
\end{enumerate}
\end{lemma}
\begin{proof}
For the proof, we abbreviate $H:=H_{\text{conn}}(G,k)$.
\begin{enumerate}
    \item Let $W\subseteq V(G)$ be such that $H[W]$ is connected. Suppose towards a contradiction that $G[W]$ is not $k$-edge connected. Then there exists a set $F$ of at most $k-1$ edges in $G[W]$ such that $G[W]-F$ is disconnected. Let $X$ denote the vertex-set of a connected component in $G[W]-F$ and let $Y:=W\setminus X$. Then both $X$ and $Y$ are non-empty, and there are at most $k-1$ edges in $G[W]$ with endpoints in $X$ and $Y$. Since $H[W]$ is connected, there has to exist a hyperedge $e$ in $H$ with $e\subseteq W=X\cup Y$ such that $e \cap X\neq \emptyset \neq e\cap Y$. Then $G[e]$ is a $k$-edge connected induced subgraph of $G[W]$ containing vertices of both $X$ and $Y$. However, since $X$ and $Y$ can be separated in $G[W]$ by deleting at most $k-1$ edges, the same is true for $G[e]$, a contradiction to the fact that $G[e]$ is $k$-edge connected. This contradiction proves the claim, $G[W]$ must indeed also be $k$-edge connected.
    \item Suppose $\chi(G)> k$. Let $G' \subseteq G$ be a minimal subgraph of $G$ with chromatic number greater than $k$. Then clearly $G'$ is $(k+1)$-color critical. It is a well-known result, compare e.g.~\cite{toft}, that every $(k+1)$-color-critical graph is $k$-edge-connected. Hence, $G'$ is $k$-edge-connected. Denoting by $e:=V(G')$ its vertex-set, we easily see that also $G[e]$ must be $k$-edge-connected, that is, $e$ forms a hyperedge of $H$. 
    \item Let $c:V(H)\rightarrow S$ be a proper coloring of $H$ for some color-set $S$ of size $\chi(H)$. Then for every $s \in S$ the color-class $c^{-1}(s)$ forms an independent set in $H$. In other words, $H[c^{-1}(s)]=H_{\text{conn}}(G[c^{-1}(s)],k)$ contains no hyperedges. Using item (2), this implies that $\chi(G[c^{-1}(s)])\le k$ for every $s \in S$. Thus,
    $$\chi(G)\le\sum_{s \in S}{\chi(G[c^{-1}(s)])} \le k|S|=k\chi(H),$$ as desired.
\end{enumerate}
\end{proof}

Using Lemma~\ref{lemma:kconn}, Theorem~\ref{thm:kconn} now falls out as an immediate consequence of our hypergraph result Corollary~\ref{smallbounds}. 
\begin{proof}[Proof of Theorem~\ref{thm:kconn}]
By Corollary~\ref{smallbounds} there exists a constant $c>0$ such that $h(t) \le c t \log \log t$ for $t \ge 3$. Let now $G$ be any given graph, $k \ge 1$ and $t \ge 3$ integers. 

Let $C:=c$ and suppose $\chi(G)>C k t\log\log t$. Then by Lemma~\ref{lemma:kconn} $\chi(H_\text{conn}(G,k)) \ge \frac{1}{k}\chi(G)>c t \log\log t \ge h(t)$. This means that $H_\text{conn}(G,k)$ contains $K_t$ as a minor, i.e., there is a $K_t$-minor model in $H_\text{conn}(G,k)$ with branch-sets $(B_i)_{i=1}^{t}$. Then for every $i \in [t]$, the set $B_i$ is connected in $H_\text{conn}(G,k)$ and thus $G[B_i]$ is $k$-edge-connected by Lemma~\ref{lemma:kconn}, (1). Now consider distinct $i, j \in [k]$. Then by definition of a minor model there exists a hyperedge $e \subseteq B_i \cup B_j$ of $H_\text{conn}(G,k)$ such that $e \cap B_i \neq \emptyset \neq e \cap B_j$. But then, since $G[e]$ is a $k$-edge-connected graph, at least $k$ edges must be spanned between $B_i$ and $B_j$ in $G$, and this proves the claim of the theorem.
\end{proof}

\subsection{Complete minors with modulo $q$-connected branch-sets}
To derive Theorem~\ref{thm:qconn}, we use the following definition of a hypergraph. For any graph $G$ and any number $q \ge 2$, we denote by $H_{\text{mod}}(G,q)$ the hypergraph with vertex-set $V(G)$ and in which a subset $e \subseteq V(G)$ of size at least $2$ is a hyperedge if and only if $G[e]$ is a modulo $q$-connected graph. Similar to Lemma~\ref{lemma:kconn} from the previous section, the following lemma relates the coloring and connectivity properties of $H_{\text{mod}}(G,q)$ and $G$.

\begin{lemma}\label{lemma:qconn}
Let $G$ be a graph and $q \ge 2$. Then
\begin{enumerate}
    \item If $W \subseteq V(G)$ is connected in $H_{\text{mod}}(G,q)$, then $G[W]$ is modulo $q$-connected. 
    \item If $\chi(G)>315q$ then $H_{\text{mod}}(G,q)$ contains at least one hyperedge. 
    \item $\chi(G)\le 315q\cdot\chi(H_\text{conn}(G))$.
\end{enumerate}
\end{lemma}

In the proof of the above Lemma, we need two statements from the literature, the first relating chromatic number and vertex-connectivity, and the second relating vertex-connectivity and modulo $q$-connectivity of graphs. 
\begin{theorem}[Gir\~{a}o and Narayanan~\cite{girao}]\label{thm:girao}
Let $k \ge 1$ be an integer. Every graph $G$ with $\chi(G)>7k$ contains a subgraph $G'$ such that $\chi(G')\ge k$ and $G'$ is $k$-vertex-connected.
\end{theorem}
We remark that the constant in the previous theorem was recently improved from $7$ to $3+\frac{1}{16}$ by Nguyen~\cite{nguyen}. For simplicity, we stick with the constant $7$ here.
\begin{theorem}[Chen, Chen, Gao and Hu~\cite{chen}]\label{thm:linked}
Let $k$ and $m_1,\ldots,m_k$ be positive integers, and let $G$ be a $45(m_1+\ldots+m_k)$-vertex-connected graph. Then at least one of the following holds: 
\begin{itemize}
    \item There exists a vertex-set $X \subseteq V(G)$ of size $|X| \le 4k-2$ such that $G-X$ is bipartite, or
    \item for every choice of pairwise distinct vertices $x_1,\ldots,x_k,y_1,\ldots,y_k$ in $G$ and for every choice of integers $d_1,\ldots,d_k$ there exist $k$ disjoint paths $P_1,\ldots,P_k$ in $G$ such that $P_i$ has endpoints $x_i$ and $y_i$, and its length is congruent to $d_i$ modulo $2m_i$. 
\end{itemize}
\end{theorem}

Setting $k:=1$ in Theorem~\ref{thm:linked}, we obtain the following special case. 
\begin{corollary}\label{cor:modconn}
Let $q \ge 2$ be an integer. Then every $45q$-vertex-connected graph $G$ contains a set $X$ of at most $2$ vertices such that $G-X$ is bipartite, or $G$ is modulo $2q$-connected (and thus also modulo $q$-connected).
\end{corollary}

We are now prepared for the proof of Lemma~\ref{lemma:qconn}.

\begin{proof}
During the proof we use the abbreviation $H=H_{\text{mod}}(G,q)$.
\begin{enumerate}
    \item Let $W \subseteq V(G)$ be connected in $H_\text{mod}(H,q)$. Let a pair of distinct vertices $x,y \in W$ as well as a residue $r\in \{0,1,\ldots,q-1\}$ be given to us. We have to show that there exists a path in $G[W]$ with endpoints $x,y$ and whose length is congruent to $r$ modulo $q$. Since $H[W]$ is connected, for some $k \ge 2$ there exists a sequence $x=x_1,x_2,\ldots,x_k=y$ of vertices in $W$ and hyperedges $e_1,\ldots,e_{k-1}$ of $H[W]$ such that $x_i, x_{i+1} \in e_i$ for all $i \in [k-1]$. We can further assume that $k\ge 2$ is chosen minimal such that a sequence as above exists.
    If $k=2$, then $x$ and $y$ are distinct vertices in the modulo $q$-connected induced subgraph $G[e_1]$ of $G[W]$, and thus can be connected by a path of length $r$ modulo $q$, as desired. 
    So assume $k \ge 3$. Then our minimality assumption on $k$ implies that $x_1 \notin e_2\cup \cdots\cup e_{k-1}$. Note that $H[e_2\cup \cdots\cup e_{k-1}]$ forms a connected subhypergraph of $H$, and hence we may apply Lemma~\ref{lemma:kconn}, (1), to find that $G[e_2\cup \cdots\cup e_{k-1}]$ is a connected subgraph of $G[W]$. Let $Q$ be a shortest path in $G[e_2\cup \cdots\cup e_{k-1}]$ that connects $y=x_k$ to a vertex in $e_1$ (such a path exists, since $x_2 \in e_1 \cap e_2$). Let $z$ denote the other endpoint of $Q$. Then we have $V(Q)\cap e_1=\{z\}$ since $Q$ is shortest. Furthermore, $x\neq z$, since $x \notin e_2\cup \cdots\cup e_{k-1}$. Let $r' \in \{0,1,\ldots,q-1\}$ be such that $r'+\ell(Q) \equiv_q r$. Then we find a path $R$ in $G[e_1]$ connecting $x$ and $z$ of length congruent to $r'$ modulo $q$. It is now apparent that $R \cup Q$ forms a path in $G[W]$ connecting $x$ and $y$ of length $\ell(R)+\ell(Q)\equiv_q r'+\ell(Q) \equiv_q r$. This concludes the proof that $G[W]$ is modulo $q$-connected.
    \item Since $\chi(G)>315q=7 \cdot 45q$, we may apply Theorem~\ref{thm:girao} and find a $45q$-vertex-connected subgraph $G'$ of $G$ with $\chi(G')\ge 45q$. W.l.o.g. we may assume that $G'$ is an induced subgraph of $G$. Applying Corollary~\ref{cor:modconn}, we find that there is a set $X$ of at most $2$ vertices such that $G'-X$ is bipartite, or $G'$ is modulo $q$-connected. However, the first option is not feasible, since $\chi(G'-X) \ge \chi(G')-|X| \ge 45q-2>2$ for every set $X\subseteq V(G')$ of size $2$. Hence, $V(G')$ forms a hyperedge in $H_\text{mod}(G,q)$, proving the assertion.
    \item The proof is analogous to the proof of item (3) of Lemma~\ref{lemma:kconn}, and is thus left out.
\end{enumerate}
\end{proof}

Similar as in the previous subsection, we can now derive Theorem~\ref{thm:qconn} as a direct consequence of Corollary~\ref{smallbounds}. 
\begin{proof}[Proof of Theorem~\ref{thm:qconn}]
Let $c>0$ be as guaranteed by Corollary~\ref{smallbounds} such that $h(t) \le c t \log \log t$ for $t \ge 3$. Let $G$ be any given graph, $q \ge 2$ and $t \ge 3$ integers. Define $C:=315c$ and suppose that $\chi(G)>C q t\log\log t$. Then by Lemma~\ref{lemma:kconn} $\chi(H_\text{mod}(G,q)) \ge \frac{1}{315q}\chi(G)>c t \log\log t \ge h(t)$. Therefore $H_\text{mod}(G,q)$ contains $K_t$ as a minor. Let $(B_i)_{i=1}^{t}$ be corresponding branch-sets of a minor model, then for every $i \in [t]$ the set $B_i$ is connected in $H_\text{mod}(G,q)$ and thus $G[B_i]$ is modulo $q$-connected by Lemma~\ref{lemma:qconn}, (1). For distinct $i, j \in [k]$ we have by definition of a minor model that there is a hyperedge $e \subseteq B_i \cup B_j$ with $e \cap B_i \neq \emptyset \neq e \cap B_j$. Since $G[e]$ is connected, this means there is at least one edge in $G$ between $B_i$ and $B_j$, and hence the set collection $(B_i)_{i=1}^t$ indeed describes a $K_t$-minor model in $G$, as desired.
\end{proof}

\subsection{Strong complete minors in digraphs}

To prove Theorem~\ref{thm:strong}, it will be crucial to relate strong minors in digraphs with the complete minors in their \emph{cycle hypergraphs}. Given a digraph $D$, its \emph{cycle hypergraph} $\mathcal{C}(D)$ has the same vertex-set $V(D)$ as $D$, and a set $e \subseteq V(D)$ forms a hyperedge of $\mathcal{C}(D)$ if and only if there exists a directed cycle $C$ in $D$ whose vertex-set equals $e$. It is clear by definition that a set $W$ of vertices is independent in $\mathcal{C}(D)$ if and only if $D[W]$ is acyclic. This directly yields that $\vec{\chi}(D)=\chi(\mathcal{C}(D))$ for every digraph $D$. The next lemma shows that complete minors in $\mathcal{C}(D)$ induce strong complete minors in $D$.

\begin{lemma}\label{lemma:stronghyprelation}
Let $t \ge 1$ be an integer. If $\mathcal{C}(D)$ contains $K_t$ as a minor, then $D$ contains $\bivec{K_t}$ as a strong minor.
\end{lemma}
\begin{proof}
Let $(B_i)_{i=1}^{t}$ be the branch-sets of a $K_t$-minor model in $\mathcal{C}(D)$. Then for every $i \in [t]$, the subhypergraph $\mathcal{C}(D)[B_i]$ is connected. We claim this means that $D[B_i]$ is strongly connected. Indeed, if not then there is a partition of $B_i$ into non-empty sets $X, Y$ such that no arc in $D$ starts in $X$ and ends in $Y$. But then no directed cycle in $D[B_i]$ can intersect both $X$ and $Y$, contradicting the connectivity of $\mathcal{C}(D)[B_i]$.

Next we claim that for every distinct $i, j \in [t]$ there is an arc from $B_i$ to $B_j$, and an arc from $B_j$ to $B_i$ in $D$. Indeed, there exists a hyperedge $e$ in $\mathcal{C}(D)$ such that $e \subseteq B_i \cup B_j$ and $e \cap B_i \neq \emptyset \neq e \cap B_j$. Then there is a  directed cycle $C$ in $D$ with $V(C)=e$, and hence $C$ must traverse an arc starting in $B_i$ and ending in $B_j$, and vice versa. 

With the above observations we see that $(B_i)_{i=1}^{t}$ certify the existence of a strong $\bivec{K_t}$-minor of $D$, as desired.
\end{proof}

The proof of Theorem~\ref{thm:strong} is now immediate using Theorem~\ref{thm:main}.
\begin{proof}[Proof of Theorem~\ref{thm:strong}]
Let $D$ be a digraph with no strong $\bivec{K_t}$-minor. By Lemma~\ref{lemma:stronghyprelation} the cycle hypergraph $\mathcal{C}(D)$ is $K_t$-minor free, and Theorem~\ref{thm:main} implies $\vec{\chi}(D)=\chi(\mathcal{C}(D)) \le h(t)\le 2g(t)$.
\end{proof}

We conclude this section with the lower-bound construction for the dichromatic number of digraphs with no strong $\bivec{K_t}$-minor. 

\begin{proof}[Proof of Proposition~\ref{prop:digraphlower}]
Let $t \ge 2$ be given. Let $A, B, C$ be $3$ vertex-disjoint sets of size $t-1$, and let $D$ be the digraph on the vertex-set $A \cup B \cup C$ with the following arcs: For every pair of distinct vertices $u, v$ that are contained together in one of $A$, $B$ or $C$, we add both the arcs $(u,v)$ and $(v,u)$. Finally, we add all arcs from $A$ to $B$, from $B$ to $C$ and from $C$ to $A$. 

For $\vec{\chi}(D) \ge \left\lceil\frac{3}{2}(t-1)\right\rceil$ it suffices to note that every set of more than $2$ vertices spans a directed cycle in $D$, and hence at least $\frac{3(t-1)}{2}$ colors are needed in any coloring with acyclic color classes.

Next, suppose towards a contradiction that $D$ contains $\bivec{K_t}$ as a strong minor. Then there are disjoint non-empty sets $(B_i)_{i=1}^{t}$ such that each of $D[B_i], i \in [t]$ is strongly connected and such between any two sets $B_i, B_j$ there exist arcs in both directions. These conditions imply that each $B_i$ is either contained in one of $A, B, C$ or intersects each of $A, B$ and $C$ in at least one vertex (and thus has size at least $3$ in this case). Since $D$ contains $3(t-1)<3t$ vertices, at least one of the sets has to be fully contained in either $A, B$ or $C$. W.l.o.g. assume that $B_1 \subseteq A$. Since $|A|=t-1$, there must exist $i \ge 2$ such that $B_i \cap A=\emptyset$. Then by the above we have $B_i \subseteq B$ or $B_i \subseteq C$. But then $B_1$ and $B_i$ are not connected by arcs in both directions, a contradiction. This concludes the proof.
\end{proof}

\end{document}